\documentclass[preprint,12pt]{elsarticle}
\usepackage{amssymb}
\usepackage{amsmath}
\usepackage{amsthm}
\usepackage{amsfonts}
\newtheorem{theorem}{Theorem}
\newtheorem{proposition}[theorem]{Proposition}
\newtheorem{corollary}[theorem]{Corollary}
\theoremstyle{definition}
\newtheorem{definition}{Definition}
\newtheorem{example}{Example}
\newtheorem{remark}{Remark}

\allowdisplaybreaks[1]
\begin{document}

\begin{frontmatter}

\title{Nonuniform power instability and Lyapunov sequences}

\author[adresaaa]{Ioan-Lucian Popa}
\ead{lucian.popa@uab.ro}
\author[adresaa]{Traian Ceau\c su}
\ead{ceausu@math.uvt.ro}
\author[adresa,adresaa]{Mihail Megan}
\ead{megan@math.uvt.ro}

\address[adresa]{Academy of Romanian Scientists, Independent\c ei
54, 050094 Bucharest, Romania}
\address[adresaa]{Department of Mathematics,
Faculty of Mathematics and Computer Science, West University of
Timi\c soara, V.  P\^ arvan Blv. No. 4, 300223-Timi\c soara,
Romania}
\address[adresaaa]{Department of Mathematics, University "$1^{st}$ December 1918" of Alba Iulia, 510009 Alba Iulia, Romania}

\begin{abstract}
The aim of this paper is to present necessary and sufficient
conditions for nonuniform power instability property of linear
discrete-time systems in Banach spaces. A characterization of the
nonuniform power instability in terms of Lyapunov sequences is
also given.
\end{abstract}
\begin{keyword} uniform power instability \sep
nonuniform power instability, Lyapunov sequences


 \MSC[2010] 34D05, 39A05
\end{keyword}
\end{frontmatter}

 \begin{flushright}
The development of mathematical dynamical systems\\
 theory can be viewed as the simultaneous pursuit\\
of two lines of research: on the one hand, the quest\\
for simplicity, comprehensibility, stability; on the other\\
 hand, the discovery of complexity,instability, chaos.\\
 Morris W. Hirsch, \cite{hirsch}, pag. 26
\end{flushright}

\section{Introduction}

The instability problem together with the stability and dichotomy
problems became one of special interests in the field of the
asymptotic behavior of linear discrete-time systems. In this
context, there are different characterizations of instability in
the papers due to N. van Minh, F. R$\ddot{a}$biger, and R.
Schnaubelt \cite{vanminh}, M. Megan, A. L. Sasu and B. Sasu
\cite{megan1}-\cite{megan2}, R. Naulin and C. J. Vanegas
\cite{naulin}, A.L. Sasu \cite{sasu}, V.E. Slyusarchuk
\cite{slyusarchuk}, B.-G. Wang and Z.-C. Wang \cite{wang}.

In \cite{sasu} A.L. Sasu obtained so-called theorems of Perron
type for exponential instability of one parameter semigroups. This
method is generalized and applied by M. Megan, A. L. Sasu and B.
Sasu in \cite{megan1} for the study of evolution operators and in
\cite{megan2} for the linear skew-product flows. Also in
\cite{megan1} generalizations to the nonuniform case of some
results of N. van Minh, F. Rabiger, and R. Schnaubelt
\cite{vanminh} are obtained. In \cite{wang} B.-G. Wang and Z.-C.
Wang characterize instability from the hyperbolic point of view.

The importance of Lyapunov functions (sequences) is well established
in the study of linear and nonlinear systems in both continous and
discrete-time. Thus, after the seminar work of A. M. Lyapunov
(republished in 1992 \cite{lyapunov}) relevant results using the
Lyapunov's direct method are presented in the books due to J.
LaSalle and S. Lefschetz \cite{lasalle}, W. Hahn \cite{hahn}, A.
Halanay and D. Wexler \cite{halanay} for linear discrete systems and
in the review of A.A. Martynyuk \cite{martynyuk} for nonlinear
discrete systems. In \cite{rugh} (Theorem 23.6) W. Rugh present the
existence of a quadratic Lyapunov sequence in order to develop
instability criteria for finite-dimensional case. This result it is
extended by J.J. DaCunha \cite{dacunha} for the timescale case.

This paper is organized as follows. In Section 2 and Section 3 we
focus our attention on some properties of uniform and nonuniform
power instability of linear discrete-time systems. Thus, we
establish relations between these concepts and we give some theorems
of characterizations for these concepts of linear discrete-time
systems in Banach spaces. In Section 4 we introduce the notion of
Lyapunov sequence and we show how nonuniform power instability can
be characterized in terms of Lyapunov sequences.

\section{Power instability}
Let $X$ be a real or complex Banach space. We consider the linear
discrete-time system

\begin{equation*}\tag{$\mathfrak{A}$}\label{A}
x_{n+1}=A{(n)}x_{n},\;\;\text{for all}\;\;n\in\mathbb{N},
\end{equation*}
where $x_n\in X$ and the operators $A(n)$ belong to
$\mathcal{B}(X),$ the space of all bounded linear operators on
$X.$ Throughout the paper, the norm on $X$ and on $\mathcal{B}(X)$
will be denoted by $\parallel . \parallel.$

We recall that if
$$\Delta=\{(m,n)\in\mathbb{N}: m\geq n\}$$
then the solution $x=(x_n)$ of (\ref{A}) is given by
$x_m=\mathcal{A}(m,n)x_n,$ for all $(m,n)\in\Delta$ where
$\mathcal{A}:\Delta\rightarrow\mathcal{B}(X)$ is defined by
 \begin{equation}\label{eqAmn}
\mathcal{A}(m,n)=\left\{\begin{array}{ll}
A{(m-1)}\cdot \ldots\cdot  A{(n)},\;\; \;\; m \geq n+1\\
\qquad\quad I\qquad\qquad\quad,\;\;\;\;\; m=n.
\end{array}\right.
 \end{equation}
 Clearly, $\mathcal{A}(m,n)\mathcal{A}(n,p)=\mathcal{A}(m,p),$ for
 all $(m,n),(n,p)\in\Delta.$

\begin{definition}\label{D: upis}
The linear discrete-time system  (\ref{A}) is said to be
 {\it uniformly power instable} (UPIS) if there are some
 constants $N\geq 1$ and $r\in (0,1)$ such that:
 \begin{equation}\label{upis}
 \parallel \mathcal{A}(n,p)x\parallel\; \leq N r^{m-n}\parallel
 \mathcal{A}(m,p) x\parallel,
 \end{equation}
 for all $(m,n),(n,p)\in \Delta$ and all $x \in X.$
\end{definition}
\begin{definition}\label{D: npis}
The linear discrete-time system  (\ref{A}) is said to be
 {\it nonuniformly power instable} (NPIS) if there exist a nondecreasing sequence $N:\mathbb{N}\longrightarrow [1,\infty)$ and $r\in
(0,1)$ such that:
 \begin{equation}\label{npis}
 \parallel \mathcal{A}(n,p)x\parallel\; \leq N(m) r^{m-n}\parallel
 \mathcal{A}(m,p) x\parallel,
 \end{equation}
 for all $(m,n),(n,p)\in \Delta$ and all $x \in X.$
\end{definition}
Two particular cases of nonuniform power instability are
introduced by
\begin{definition}\label{D: pis}
The linear discrete-time system  (\ref{A}) is said to be
\begin{description}

\item[{\it{i)}}] {\it power instable} PIS if there are some
constants $N \geq 1,$ $r\in(0,1)$ and $s\geq 1$ such that:
\begin{equation}\label{eis}
\parallel \mathcal{A}(n,p)x\parallel\; \leq N r^{m-n}s^{ n}
\parallel \mathcal{A}({m},{p})x\parallel,
\end{equation}
 for all $(m,n),(n,p)\in \Delta$ and all $x \in X.$

\item[{\it{ii)}}] {\it strongly power instable} SPIS if there are
some constants $N \geq 1,$ $r\in (0,1)$ and $s\in \left[1,
\dfrac{1}{r}\right)$ such that:
\begin{equation}\label{ses}
\parallel \mathcal{A}(n,p)x\parallel\; \leq N r^{m-n}s^{n} \parallel
\mathcal{A}({m},{p})x\parallel,
\end{equation}
 for all $(m,n),(n,p)\in \Delta$ and all $x \in X.$
\end{description}
\end{definition}
It is obvious that if  system (\ref{A}) is UPIS than it is NPIS.
The following example shows that the converse implication is not
valid.
\begin{example}\label{E: npisNOTupis}
Let $c\in\mathbb{R}_{+}^{*},$ $b\geq 2$ and (\ref{A}) be the
linear-time system defined for all $n\in\mathbb{N}$ by
$$A_{n} =c\cdot a_{n}I,\;\;\text{where}\;\;
  a_{n} = \left\{
  \begin{array}{l l}
    b^{-n} & \quad \text{if $n=2k$}\\
    b^{n+1} & \quad \text{if $n=2k+1$}\;\;.\\
  \end{array} \right.$$
  The following statements are true:

  \begin{description}

\item[{\it{i)}}](\ref{A}) is not UPIS for all
$c\in\mathbb{R}_{+}^{*}$;

\item[{\it{ii)}}] if $c>1,$ then (\ref{A}) is NPIS.
\end{description}

Let $(m,n,x)\in\Delta\times X.$ According to (\ref{eqAmn}) we have
that
\[
  \mathcal{A}({m},{n})x = \left\{
  \begin{array}{l l}
    c^{m-n}a_{mn}x & \quad m>n\\
    x & \quad m=n\\
  \end{array} \right. ,
\]
 where
\[
  a_{mn} = \left\{
  \begin{array}{l l}
    b^{m-n} & \quad \text{if $m=2q$ and $n=2p$}\\
    b^{m} & \quad \text{if $m=2q$ and $n=2p+1$}\\
    b^{-n} & \quad \text{if $m=2q+1$ and $n=2p$}\\
    1 & \quad \text{if $m=2q+1$ and $n=2p+1$}\\
  \end{array} \right.
\]
$(i)$ If we suppose that (\ref{A}) is UPIS then there exist some
constants $N \geq 1$ and $r\in(0,1)$ such that
$$\parallel x\parallel\; \leq N r^{m-n} \parallel \mathcal{A}({m},{n})x\parallel=
N (rc)^{m-n}a_{mn}\parallel x\parallel$$
 for all $(m,n,x)\in\Delta\times X,$ which  is equivalent with
\[
  \left\{
  \begin{array}{l l}
    \left( \dfrac{1}{rc}\right)^{m-n}b^{-m+n} \leq N & \quad \text{if $m=2q$ and $n=2p$}\\
    \left( \dfrac{1}{rc}\right)^{m-n}b^{-m} \leq N & \quad \text{if $m=2q$ and $n=2p+1$}\\
    \left( \dfrac{1}{rc}\right)^{m-n}b^{n} \leq N & \quad \text{if $m=2q+1$ and $n=2p$}\\
    \left( \dfrac{1}{rc}\right)^{m-n} \leq N & \quad \text{if $m=2q+1$ and $n=2p+1$}.\\
  \end{array} \right.
\]
There are two cases that can be considered at this point. If $c\in
(0,1]$ then for  $m=2q+1$ and $n=2p+1\in \mathbb{N}$ fixed we have
that
\begin{equation}\label{ex1eq1}
\lim\limits_{q\rightarrow\infty} \left(
\frac{1}{rc}\right)^{2q-2p}=\infty.
\end{equation}
If $c\in (1,\infty)$,  $n=2p$ and $m=n+1$ it follows that
\begin{equation}\label{ex1eq2}
\lim\limits_{p\rightarrow\infty} \left(
\frac{1}{rc}\right)b^{2p}=\infty.
\end{equation}
According to (\ref{ex1eq1}) and (\ref{ex1eq2}) we can conclude
that (\ref{A}) can not be UPIS.

$(ii)$ If $c>1$ then for $r=\dfrac{1}{cb}$ and $N(m)=b^{m}$ the
system (\ref{A}) is NPIS.
\end{example}
\begin{remark}
Previous Example (with $b=2$) was studied in \cite{popa5} in order
to prove that the concepts of PIS and NPIS are not equivalent.
Thus, the system (\ref{A}) is PIS if and only if $c\geq 1$ and
SPIS if and only if $c\geq 2.$
\end{remark}
\begin{proposition}\label{P: mneis}
The linear discrete-time system (\ref{A}) is  NPIS if and only if
there are two nondecreasing sequences
$\varphi,\tau\,{:}\,\mathbb{N}\longrightarrow [1,\infty)$ with
$\lim\limits_{n\rightarrow\infty}\,\tau(n){=}\infty$ such that
\begin{equation}\label{eq1 P:mneis}
{\tau(m-n)}\parallel{x}\parallel\;\leq\varphi(m)\parallel{\mathcal{A}(m,n)x}\parallel,
\end{equation}
 for all $(m,n,x)\in\Delta\times X$.
\end{proposition}
\begin{proof}
{\it Necessity.} It is a simple verification for $\varphi(m)=N(m)$
and $\tau(m)=r^{-m}$ where $r\in (0,1)$ is given by Definition
\ref{D: npis}.

{\it Sufficiency.} Using the hypothesis
$\lim\limits_{n\rightarrow\infty}\tau (n)=\infty$ we have that
there exists $c\in\mathbb{N}^{*}$ with $\tau (c)>1.$ Then for all
$(m,n,x)\in \Delta\times X$ there exist two constants
$k\in\mathbb{N}$ and $s\in\{0,1,\ldots ,c-1\}$ such that
$m=n+kc+s.$ According to (\ref{eq1 P:mneis}) we have that
\begin{align*}
\parallel\mathcal{A}({m},{n})x\parallel
&\geq\dfrac{\tau
(s)}{\varphi (m)}\parallel\mathcal{A}({n+kc},{n})x\parallel\;,\\
\parallel\mathcal{A}({n+kc},{n})x\parallel&\geq\dfrac{\tau
(c)}{\varphi
(n+kc)}\parallel\mathcal{A}({n+(k-1)c},{n})x\parallel\;,\\
\parallel\mathcal{A}({n+(k-1)c},{n})x\parallel
&\geq\dfrac{\tau (c)}{\varphi
(n+(k-1)c)}\parallel\mathcal{A}({n+(k-2)c},{n})x\parallel\;,\\
&\vdots\;\;,\\
\parallel\mathcal{A}({n+2c},{n})x\parallel
&\geq\dfrac{\tau (c)}{\varphi
(n+2c)}\parallel\mathcal{A}({n+c},{n})x\parallel\;,\\
\parallel\mathcal{A}({n+c},{n})x\parallel
&\geq\dfrac{\tau (c)}{\varphi (n+c)}\parallel x\parallel\;,
\end{align*}
which implies
\begin{equation*}
\parallel\mathcal{A}({m},{n})x\parallel\;\geq\dfrac{\tau (s)\tau
(c)^{k}}{\varphi (m)\cdot\varphi (n+kc)\cdot\ldots\cdot\varphi
(n+c)}\parallel x\parallel.
\end{equation*}
Furthermore, since
\begin{equation*}
\varphi (m)\geq \varphi (n+kc)\geq\ldots\geq \varphi
(n+2c)\geq\varphi (n+c),
\end{equation*}
we obtain
\begin{align*}
\varphi (m)^{k+1}\parallel\mathcal{A}({m},{n})x\parallel &\geq
\tau (0)\tau (c)^{k}\parallel x\parallel\\
&=\tau (0) r^{-(m-n)}r^{s}\parallel x\parallel,
\end{align*}
where $r=e^{-\frac{\ln \tau (c)}{c}}.$ Thus, we can conclude that
(\ref{A}) is NPIS.
\end{proof}
%
A similar result with Proposition \ref{P: mneis} for the concept
of UPIS is given by
\begin{corollary}\label{C: upis}
The linear discrete-time system (\ref{A}) is UPIS if and only if
there exists a nondecreasing sequence
$\varphi:\mathbb{N}\longrightarrow\mathbb{R}_{+}^{*}$ with
$\lim\limits_{n\rightarrow\infty}\varphi(n)=\infty$ such that
\begin{equation*}
\varphi{(m)}\parallel x\parallel \;\leq\; \parallel
\mathcal{A}(m+n,n)x\parallel,\;\;\text{for all}\;\;
(m,n)\in\mathbb{N}^{2}\;\;\text{and all}\;\;x\in X.
\end{equation*}
\end{corollary}
\begin{remark}\label{R: 2upis}
For time-invariant linear systems, (i.e. $A(n)=A$ and
$\mathcal{A}(m,n)=A^{m-n}$) we have that
$\lim\limits_{n\rightarrow\infty}\parallel A^{n}\parallel
\;=\infty$ is a necessary and sufficient condition for the concept
of UPIS.

In the following example we show that for time-varying linear
systems the relation
\begin{equation}\label{eq inf}
\lim\limits_{m\rightarrow\infty}\parallel
\mathcal{A}(m,n)\parallel \;=\infty
\end{equation}
represents just a necessary condition for UPIS.
\end{remark}
\begin{example}\label{E: upis}
Let (\ref{A}) be the linear-time system given by
\begin{equation}\label{eq A:E upis}
A(n)=\dfrac{u(n+1)}{u(n)}I,
\end{equation}
 where $u(n)=n+2,$ for all
$n\in\mathbb{N}.$ According to (\ref{eqAmn}) we have that
\begin{equation*}
  \mathcal{A}(m,n)x\;=\;\dfrac{m+2}{n+2}x,\;\;\text{for
  all}\;\;(m,n,x)\in\Delta\times X.
\end{equation*}
Firstly, we observe that
$\lim\limits_{m\rightarrow\infty}\parallel
\mathcal{A}(m,n)\parallel\; =\infty.$ On the other hand, if we
suppose that (\ref{A}) is UPIS then there exists $N\geq 1$ and
$r\in (0,1)$ such that
\begin{equation*}
r^{n}\leq
N\;\dfrac{m+n+2}{n+2}\;r^{m}=\dfrac{N}{n+2}\;m\;r^{m}+N\;r^{m}.
\end{equation*}
From here, for fixed $n\in\mathbb{N}$ and $m\rightarrow\infty$ we
have that $0<r^{n}\leq 0$ which is false.

\end{example}
Note that the property (\ref{eq inf}) is not valid for the concept
of NPIS. This fact is illustrated by
\begin{example}\label{E: npis}
Let (\ref{A}) be the discrete-time system given by (\ref{eq A:E
upis}), with $u(n)=e^{-n},$ for all $n\in\mathbb{N}.$ Thus, we
obtain
\begin{equation*}
  \mathcal{A}(m,n)x=e^{n-m}x,\;\;\text{for all}\;\;(m,n,x)\in\Delta \times X.
\end{equation*}
Obvious, (\ref{eq inf}) it is false. Furthermore, the system
(\ref{A}) is not UPIS. Finally, we remark that for
$N(n)=e^{n^{2}}$ and $r=e^{-2}$ the inequality (\ref{npis}) is
satisfied and thus the system (\ref{A}) is NPIS.
\end{example}
We can observe that the system considered in Example \ref{E: npis}
is uniformly exponentially stable. For further information about
(non)uniform exponential stability concepts see \cite{popa4}.
There are some particular concepts of NPIS for what the property
(\ref{eq inf}) holds (see \cite{popa5}).
\section{Auxiliary results.}
\begin{theorem}\label{T: NPIS}
The linear discrete-time system (\ref{A}) is NPIS if and only if
there are some constants $d> 1,$ $p\in (0,+\infty)$ and a
nondecreasing sequence $\theta :\mathbb{N}\rightarrow [1,+\infty)$
such that
\begin{equation}\label{eq T: npis}
\sum\limits_{k=n}^{m}d^{p({m-k})}\parallel
\mathcal{A}(k,n)x\parallel^{p}\;\leq \theta (m)^{p}\parallel
\mathcal{A}(m,n)x\parallel^{p},
\end{equation}
for all $(m,n,x)\in\Delta\times X.$
\end{theorem}
\begin{proof}
{\it Necessity.} Using Definition \ref{D: npis} we have that
\begin{align*}
\sum\limits_{k=n}^{m}d^{p(m-k)}\parallel\mathcal{A}(k,n)x\parallel^{p}
&\leq
\sum\limits_{k=n}^{m}d^{p(m-k)}N(m)^{p}r^{p(m-k)}\parallel\mathcal{A}(m,n)x\parallel^{p}\\
&=N(m)^{p}\parallel\mathcal{A}(m,n)x\parallel^{p}\sum\limits_{k=n}^{m}\left[(rd)^{p}\right]^{m-k}\\
&\leq\dfrac{N(m)^{p}}{1-(rd)^{p}}\parallel\mathcal{A}(m,n)x\parallel^{p}
\end{align*}
for any $d\in \left(1,\dfrac{1}{r}\right)$ and all
$(n,x)\in\mathbb{N}\times X.$

{\it Sufficiency.} The inequality (\ref{eq T: npis}) implies that
\begin{equation*}
d^{p(m-n)}\parallel x\parallel^{p}\leq \theta
(m)^{p}\parallel\mathcal{A}(m,n)x\parallel^{p},\;\;\text{for
all}\;\;(m,n,x)\in\Delta\times X.
\end{equation*}
From this it results that
\begin{equation*}
\parallel
x\parallel\leq\theta(m)\left(\dfrac{1}{d}\right)^{m-n}\parallel\mathcal{A}(m,n)x\parallel
\end{equation*}
and thus the system (\ref{A}) is NPIS.
\end{proof}
\begin{corollary}\label{P: lp-UPIS}
The linear discrete-time system (\ref{A}) is UPIS if and only if
there are some constants $p\in (0,+\infty),$ $D\geq 1,$ $d> 1$ such
that
\begin{equation}\label{eq P: UPIS}
\sum\limits_{k=n}^{m}d^{p({m-k})}\parallel
\mathcal{A}(k,n)x\parallel^{p}\;\leq D^{p}\parallel
\mathcal{A}(m,n)x\parallel^{p},
\end{equation}
for all $(m,n,x)\in \Delta\times X.$
\end{corollary}
\begin{proof}
It results as a particular case from Theorem \ref{T: NPIS}.
\end{proof}
According to the previous theorem we  can obtain necessary and
sufficient criteria for other (particular) cases of instability.
More precisely
\begin{proposition}\label{P: lp-PIS}
The linear discrete-time system (\ref{A}) is PIS if and only if
there are some constants $p\in (0,+\infty),$ $D\geq 1,$ $d>1$ and
$c> 1$ with $c\in (1,d)$ such that
\begin{equation}\label{eq P: PIS}
\sum\limits_{k=n}^{m}d^{p({m-k})}\parallel
\mathcal{A}(k,n)x\parallel^{p}\;\leq D^{p}c^{pm}\parallel
\mathcal{A}(m,n)x\parallel^{p},
\end{equation}
for all $(m,n,x)\in \Delta\times X.$
\end{proposition}
\begin{proof}
{\it Necessity.} Using Definition \ref{D: pis} $(i)$ we have that
$1<\dfrac{1}{r}\leq \dfrac{s}{r}.$ Thus, for $d>1$ with
$0<\dfrac{s}{rd}<1$ we have that
\begin{align*}
\sum\limits_{k=n}^{m}d^{p(m-k)}\parallel\mathcal{A}(k,n)x\parallel^{p}
&\leq N^p(rd)^{pm}\parallel\mathcal{A}(m,n)x\parallel^{p}\sum\limits_{k=n}^{m}\left[\left(\frac{s}{rd}\right)^{p}\right]^{k}\\
&\leq\frac{N^{p}(rd)^{pm}}{1-\left(\frac{s}{rd}\right)^{p}}\parallel\mathcal{A}(m,n)x\parallel^{p},
\end{align*}
for all $(n,x)\in\mathbb{N}\times X.$

{\it Sufficiency.} It is easy to see that, for all
$(m,n,x)\in\Delta\times X$ we have the following inequality
\begin{equation*}
\parallel x\parallel\leq
D\left(\dfrac{c}{d}\right)^{m-n}c^n\parallel\mathcal{A}(m,n)x\parallel.
\end{equation*}
Thus, we can conclude that system (\ref{A}) is PIS.
\end{proof}
\begin{proposition}\label{P: lp-SPIS}
The linear discrete-time system (\ref{A}) is SPIS if and only if
there are some constants $p\in (0,+\infty),$ $D\geq 1,$ $d>1$ and
$c> 1$ with $c^{2}<d$ such that (\ref{eq P: PIS}) hold for all
$(m,n,x)\in \Delta\times X.$
\end{proposition}
\begin{proof}
{\it Necessity.} Since $0<r<1\leq s <\dfrac{1}{r}$ we have that
$1\leq \dfrac{s}{r}<\dfrac{1}{r^2}.$ In the same manner as we proved
Proposition \ref{P: lp-PIS}, if $d>1$ such that
$\dfrac{s}{r}<d<\dfrac{1}{r^2}$ and $c=rd>s\geq 1,$ then relation
(\ref{eq P: PIS}) it is verified for all $(n,x)\in\mathbb{N}\times
X.$

{\it Sufficiency.} This follows using similar arguments to those in
the proof of Proposition \ref{P: lp-PIS}.
\end{proof}
\section{Lyapunov sequences}
\begin{definition}\label{D lnes}
We say that $L:\Delta\times X\rightarrow\mathbb{R}_{+}$ is
 a {\it Lyapunov sequence} for the  system (\ref{A}) if
there exists a constant $a\in (1,\infty)$ such that
\begin{equation*}
L(n,n,x)=\parallel x\parallel
\end{equation*}
and
\begin{equation}\label{Eq L}
L\left(m,n,x\right)-aL\left(m-1,n,x\right)\geq
\parallel \mathcal{A}(m,n)x\parallel
\end{equation}
for all $(m,n,x)\in \Delta\times X,$ with $m>n.$
\end{definition}
\begin{theorem}\label{T: Lnes}
The linear discrete-time system (\ref{A}) is NPIS if and only if
there exist a Lyapunov sequence and a nondecresing sequence $\beta
:\mathbb{N}\rightarrow [1,\infty)$ such that
\begin{equation}
L(m,n,x) \leq \beta(m)\parallel \mathcal{A}(m,n)x\parallel,
\end{equation}
for all $(m,n,x)\in\Delta\times X.$
\end{theorem}
\begin{proof}
{\it Necessity.} Let $d>1$. We define $L:\Delta\times
X\rightarrow\mathbb{R}_{+}$ by
$$L(m,n,x)=\sum\limits_{k=n}^{m}d^{m-k}\parallel \mathcal{A}(k,n)x\parallel,$$
for all $(m,n,x)\in\Delta\times X.$

First, we observe that
\begin{align*}
 L(m,n,x) &=\sum\limits_{k=n}^{m}d^{m-k}\parallel \mathcal{A}(k,n)x\parallel\\
 &=d^{m-n}\parallel x\parallel +\ldots +\parallel\mathcal{A}(m,n)x\parallel\\
 &=dL(m-1,n,x)+\parallel\mathcal{A}(m,n)x\parallel,
\end{align*}
 for all $(m,n,x)\in\Delta\times X,$ with $m>n.$ Hence
 $$L\left(m,n,x\right)-aL\left(m-1,n,x\right)\geq
\parallel \;\mathcal{A}(m,n)x\parallel$$
for every $a\in(1,d)$ and $(m,n,x)\in\Delta\times X,$ with $m>n.$

On the other hand, for $d\in \left(1,\dfrac{1}{r}\right)$ we have
that
\begin{align*}
L(m,n,x) &\leq \sum\limits_{k=n}^{m}N(m)d^{m-k}r^{m-k}\parallel
\mathcal{A}(m,n)x\parallel\\
&=N(m)\parallel\mathcal{A}(m,n)x\parallel\sum\limits_{k=n}^{m}(dr)^{m-k}\\
&\leq \dfrac{N(m)}{1-dr}\parallel\mathcal{A}(m,n)x\parallel\\
&=\beta (m)\parallel\mathcal{A}(m,n)x\parallel.
\end{align*}

{\it Sufficiency.} According to (\ref{Eq L}) we have that
\begin{align*}
L(m,n,x) -&aL(m-1,n,x)\geq\;\parallel \mathcal{A}({m},{n})x\parallel\\
L(m-1,n,x) -&aL(m-2,n,x)\geq\;\parallel \mathcal{A}({m-1},{n})x\parallel\\
\ldots\;\;\;\;\;\ldots\;\;\;\;\;&\ldots\;\;\;\;\;\ldots\;\;\;\;\;\ldots\\
L(n+1,n,x)-&aL(n,n,x)\geq\;\mathcal{A}(n+1,n,x)
\end{align*}
which implies
\begin{equation}\label{eq1 DemT}
\sum\limits_{j=n}^{m}a^{m-j}\parallel
\mathcal{A}({j},{n})x\parallel\; \leq L(m,n,x)\leq
\beta(m)\parallel \mathcal{A}(m,n)x\parallel
\end{equation}
But
\begin{equation}\label{eq12 DemT}
\sum\limits_{j=n}^{m}a^{m-j}\parallel
\mathcal{A}({j},{n})x\parallel \geq a^{m-n}\parallel x\parallel.
\end{equation}
Now, using (\ref{eq1 DemT}), (\ref{eq12 DemT}) and Proposition
\ref{P: mneis} we obtain that system (\ref{A}) is NPIS.
\end{proof}

As a consequence of the previous theorem we obtain
\begin{corollary}
The linear discrete-time system (\ref{A}) is NPIS if and only if
there exist a nondecreasing sequence $\theta
:\mathbb{N}\rightarrow [1,\infty)$ such that
\begin{equation*}
\sum\limits_{k=n}^{m}\parallel \mathcal{A}(k,n)x\parallel\leq
\theta (m)\parallel\mathcal{A}(m,n)x\parallel,
\end{equation*}
for all $(m,n,x)\in\Delta\times X.$
\end{corollary}
\begin{proof}
It results from Definition \ref{D: npis} and the proof of Theorem
\ref{T: Lnes} and Proposition \ref{P: mneis}.
\end{proof}

\begin{flushleft}
{\bf{Acknowledgments}}
\end{flushleft}

The authors would like to express their sincere thanks to the reviewer for the valuable suggestions and comments which
have led to an improved final version of the paper.

\end{document}